\documentclass[a4paper,11pt,reqno,twoside]{amsart}
\usepackage{GL_n_amplification}
\begin{document}%
%...............................................................................................
\section{Introduction and statement of the results}%
%...............................................................................................
\subsection{Motivation}%
%...............................................................................................
\subsubsection{The general philosophy of the amplification method}%
The amplification method was set up by W.~Duke, J.~Friedlander and H.~Iwaniec (see \cite{MR1137896}, \cite{MR1166510} and \cite{MR1258904} for example).
\par
When bounding say a complex number $z$, which satisfies for obvious reasons depending on the context
\begin{equation}\label{eq_trivial}
\abs{z}\leq M
\end{equation}
for some positive real number $M$ but, which is expected to satisfy
\begin{equation}\label{eq_non_trivial}
\abs{z}\leq M^{1-\delta}
\end{equation}
for some $0<\delta<1$, it is sometimes profitable to include\footnote{Note that choosing a family containing $z$ may be highly non-trivial. In particular, it should be large enough in order to be able to use the powerful tools of harmonic analysis but not too large such that bounding a moment of small order, like the second one, has a chance to be successful.} $z$ in a finite family of complex numbers of the same nature, say
\begin{equation*}
z=z_{j_0}\in\left\{z_j, j\in J\right\}\coloneqq\mathcal{Z}_J
\end{equation*}
where $J$ is a finite set of cardinality $\asymp M$, $j_0\in J$ is the index of our favourite complex number $z$ and to estimate all the quantities occuring in this family on average.
\par
For instance, one can try to bound the second moment of this family given by
\begin{equation*}
M_2\left(\mathcal{Z}_J\right)\coloneqq\sum_{j\in J}\abs{z_j}^2.
\end{equation*}
By \eqref{eq_trivial}, the second moment satisfies
\begin{equation*}
M_2\left(\mathcal{Z}_J\right)\leq \abs{J}M^2,
\end{equation*}
which does not help us to prove \eqref{eq_non_trivial} by positivity.
\par
One can try to bound instead an amplified second moment given by
\begin{equation*}
\mathcal{M}_2\left(\mathcal{Z}_J,\overrightarrow{\alpha}\right)\coloneqq\sum_{j\in J}\left\vert M_j(\overrightarrow{\alpha})\right\vert^2\abs{z_j}^2
\end{equation*}
where $M_j\left(\overrightarrow{\alpha}\right)$ is a short Dirichlet polynomial given by
\begin{equation*}
M_j\left(\overrightarrow{\alpha}\right)\coloneqq\sum_{i\in I}\alpha_ia_j(i)
\end{equation*}
for $j\in J$ and where $I$ is a small finite set. Here, $\overrightarrow{\alpha}=\left(\alpha_i\right)_{i\in I}$ is a finite sequence of complex numbers, which will be specified later on, and $\left(a_j(i)\right)_{i\in I}$ are some complex numbers naturally related to $z_j$ for $j\in J$. In practice, the currently known techniques enable us to prove
\begin{equation}\label{eq_bound_amplified}
\mathcal{M}_2\left(\mathcal{Z}_J,\overrightarrow{\alpha}\right)\leq M^\epsilon\left(M^2\abs{\abs{\overrightarrow{\alpha}}}_2^2+\abs{I}^\beta\abs{\abs{\overrightarrow{\alpha}}}_1\right)
\end{equation}
for some possibly large $\beta>0$ and for all $\epsilon>0$, where as usual $\abs{\abs{\overrightarrow{\alpha}}}_1$ stands for the $L^1$-norm of $\overrightarrow{\alpha}$ and $\abs{\abs{\overrightarrow{\alpha}}}_2$ for its $L^2$-norm.
\par
The whole point of the amplification method is to choose a sequence $\overrightarrow{\alpha}$, which amplifies the contribution of the complex number $z$ in the amplified second moment $\mathcal{M}_2\left(\mathcal{Z}_J,\overrightarrow{\alpha}\right)$. More explicitely, one has to construct a sequence $\overrightarrow{\alpha}$ satisfying\footnote{Obviously one should also expect that $\left\vert M_{j}\left(\overrightarrow{\alpha}\right)\right\vert^2$ is not too large when $j\neq j_0$ in $J$ for the amplification method to be successful. This generally follows in concrete cases, at least conditionally, from a suitable version of the Riemann Hypothesis. Hopefully, one does not this in practice.} 
\begin{eqnarray}
\abs{\abs{\overrightarrow{\alpha}}}_2 & \leq & \abs{I}^\epsilon, \\
\left\vert M_{j_0}\left(\overrightarrow{\alpha}\right)\right\vert^2 & \geq & \abs{I}^\gamma
\end{eqnarray}
for some possibly small $\gamma>0$ and for all $\epsilon>0$. In general, cooking such sequence $\overrightarrow{\alpha}$ is based on the fact that some of the complex numbers $a_{j_0}(i)$, $i\in I$, cannot be small simultaneously. For such sequence, \eqref{eq_bound_amplified} entails by positivity
\begin{equation}\label{eq_fin}
\abs{z}^2=\abs{z_{j_0}}^2\leq\left(M\abs{I}\right)^\epsilon\left(\frac{M^2}{\abs{I}^\gamma}+\abs{I}^{\beta+1/2-\gamma}\right)
\end{equation}
for all $\epsilon>0$, which implies \eqref{eq_non_trivial} by an optimal choice of $\abs{I}$.
\par
The very natural first step towards the proof of \eqref{eq_bound_amplified} is to open the square and to switch the order of summation, which leads us to bounding
\begin{equation}\label{eq_expand}
\sum_{(i_1,i_2)\in I^2}\alpha_{i_1}\overline{\alpha_{i_2}}\sum_{j\in J}a_j(i_1)\overline{a_j(i_2)}\abs{z_j}^2.
\end{equation}
The diagonal term, namely the contribution from $i_1=i_2$ in the previous equation, is generally bounded by the first term in the right-hand side of \eqref{eq_fin}, whereas the non-diagonal term, namely the contribution from $i_1\neq i_2$ in the previous equation, is generally bounded by the second term in the right-hand side of \eqref{eq_fin}.
\par
Getting these bounds heavily relies in practice on linearising the products $a_j(i_1)\overline{a_j(i_2)}$ for $i_1$ and $i_2$ in $I$, namely these products can be often written in relevant cases as a linear combination of the $a_j(i)$'s. Such linearisations in the context of $GL(n)$ automorphic forms are the core of this article.
\par
In practice, the complex numbers $a_j(i)$ and $\overline{a_j(i)}$, $(i,j)\in I\times J$, are the eigenfunctions of some specific endomorphisms. Thus, linearising the products $a_j(i_1)\overline{a_j(i_2)}$ boils down to linearising the composition of these relevant endomorphisms.
%...............................................................................................
\subsubsection{The amplification method in $GL(n)$}%
Let $p$ and $q$ be two prime numbers.
\par
In the context of $GL(n)$ automorphic forms defined in Section \ref{sec_GL3}, our favourite complex number $z$ is related to a $GL(n)$ Hecke-Maa$\ss$ cusp form $f$, say $z=z(f)$. For instance, $z=f(g)$ for $g$ in the generalised upper-half plane or $z=L(f,s)$, the value at of the Godement-Jacquet $L$-function attached to $f$ on the critical line $\Re{(s)}=1/2$. 
\par
Hence $z$ can be included, with a slight abuse of notations, in a finite subset of an orthonormal basis $(f_j)_{j\geq 1}$ of $GL(n)$ Hecke-Maa$\ss$ cusp forms, namely those whose analytic conductors, the Laplace eigenvalue or the level or the imaginary part of $s$ for instance, is bounded by some parameter $Q>0$, which is devoted to tend to infinity, say
\begin{equation*}
z(f)=z(f_{j_0})\in\left\{z(f_j), j\geq 1, Q(f_j)\leq Q\right\}.
\end{equation*}
\par
In \cite{SiVe2}, the authors proved the existence of an abstract higher rank amplifier and in \cite{BlMa1}, the authors proved that there exists, at least asymptotically ($p$ large), a non-trivial linear combination of $GL(n)$ Hecke operators equal to the identity operator (see \cite[Lemma 4.2]{BlMa1}). The whole point of this work is to give a totally explicit and ready to use version of a $GL(n)$ amplifier.
\par
The choice of our amplifier $\overrightarrow{\alpha}$ relies on the fundamental identity
\begin{equation*}
a_{j_0}(p,\underbrace{1,\dots,1}_{n-2 \textnormal{ terms}})a_{j_0}(\underbrace{1,\dots,1}_{n-2 \textnormal{ terms}},p)=a_{j_0}(p,\underbrace{1,\dots,1}_{n-3 \textnormal{ terms}},p)+1,
\end{equation*}
where $a_j(m_1,\dots,m_{n-1})$ stands for the $(m_1,\dots,m_{n-1})$'th Fourier coefficient of $f_j$ (see \eqref{eq_fourier} and \cite[Theorem 9.3.11 Page 271]{MR2254662}). This identity essentially says that $a_{j_0}(p,\underbrace{1,\dots,1}_{n-2 \textnormal{ terms}})a_{j_0}(\underbrace{1,\dots,1}_{n-2 \textnormal{ terms}},p)$ and $a_{j_0}(p,\underbrace{1,\dots,1}_{n-2 \textnormal{ terms}},p)$ cannot be simultaneously small . At the level of Hecke operators, this identity reflects the fact that
\begin{equation}\label{eq_identity}
T_{\textnormal{diag}(1,\underbrace{p,\dots,p}_{n-1 \textnormal{ terms}})}\circ T_{\textnormal{diag}(\underbrace{1,\dots,1}_{n-1 \textnormal{ terms}},p)}=T_{\textnormal{diag}(1,\underbrace{p,\dots,p}_{n-2 \textnormal{ terms}},p^2)}+\frac{p^{n}-1}{p-1}\textnormal{Id},
\end{equation}
itself a consequence of the identity
\begin{multline*}
\Lambda_n\textnormal{diag}\left(\underbrace{1,\dots,1}_{n-1 \textnormal{ terms}},p\right)\Lambda_n\ast\Lambda_n\textnormal{diag}\left(1,\underbrace{p,\dots,p}_{n-1 \textnormal{ terms}}\right)\Lambda_n=\Lambda_n\textnormal{diag}\left(1,\underbrace{p,\dots,p}_{n-2 \textnormal{ terms}},p^2\right)\Lambda_n \\
+\frac{p^{n}-1}{p-1}\Lambda_n\textnormal{diag}\left(\underbrace{p,\dots,p}_{n \textnormal{ terms}}\right)\Lambda_n
\end{multline*}
at the level of $\Lambda_n$ double cosets, where $\Lambda_n\coloneqq GL_n(\Z)$ (see \cite[Lemma 2.18 Page 114]{MR1349824}).
\par
The coefficients $a_j(i)$'th will be some Hecke eigenvalues of $f_j$. More precisely, being inspired by \cite{HoRiRo2} and by \eqref{eq_identity}, we set
\begin{eqnarray*}
a_j(p) & \coloneqq & a_j(p,\underbrace{1,\dots,1}_{n-1 \textnormal{ terms}})=\textnormal{the eigenvalue of $T_p=p^{-(n-1)/2}T_{\textnormal{diag}(\underbrace{1,\dots,1}_{n-1 \textnormal{ terms}},p)}$}, \\
a_j\left(p^{2}\right) & \coloneqq & \textnormal{the eigenvalue of $p^{-(n-1)}T_{\textnormal{diag}(1,\underbrace{p,\dots,p}_{n-2 \textnormal{ terms}},p^2)}$}\in\R
\end{eqnarray*}
when acting on $f_j$ and we recall that
\begin{equation*}
\overline{a_j(p)}=\textnormal{the eigenvalue of $T_p^\ast=p^{-(n-1)/2}T_{\textnormal{diag}(1,\underbrace{p,\dots,p}_{n-1 \textnormal{ terms}})}$}
\end{equation*}
still when acting on $f_j$ (see \eqref{eq_Hecke}). Thus, $I$ is a subset of the prime numbers and of the squares of the prime numbers.
\par
A very natural candidate for a $GL(n)$ amplifier is
\begin{equation*}
M_j\left(\overrightarrow{\alpha}\right)\coloneqq\sum_{i\in I}\alpha_ia_j(i)
\end{equation*}
where
\begin{equation*}
\alpha_{i}\coloneqq\begin{cases}
\overline{a_{j_0}(p)} & \text{if $i=p\leq\sqrt{L}$ is a prime number,} \\
-1 & \text{if $i=p^2\leq L$ is the square of a prime number} \\
0 & \text{otherwise}.
\end{cases}
\end{equation*}
This amplifier satisfies, as in the $GL(2)$ and $GL(3)$ case,  $\vert M_{j_0}\left(\overrightarrow{\alpha}\right)\vert^2\gg_\epsilon L^{1-\epsilon}$ since $\vert I\vert\gg_\epsilon L^{1-\epsilon}$ for all $\epsilon>0$.
\par
Glancing at \eqref{eq_expand} and applying the inequality\footnote{Such inequality, used for the first time in the amplification method in \cite{BlHaMil}, enabled the authors to avoid mixing squares of prime numbers and prime numbers in their diophantine analysis.} 
\begin{equation*}
\left\vert M_{j_0}\left(\overrightarrow{\alpha}\right)\right\vert^2\leq 2\left\vert\sum_{p\leq\sqrt{L}}\alpha_pa_j(p)\right\vert^2+2\left\vert\sum_{p\leq\sqrt{L}}\alpha_{p^2}a_j(p^2)\right\vert^2,
\end{equation*}
it becomes crucial to linearise the products
\begin{multline*}
T_{\textnormal{diag}(1,\underbrace{p,\dots,p}_{n-1 \textnormal{ terms}})}\circ T_{\textnormal{diag}(\underbrace{1,\dots,1}_{n-1 \textnormal{ terms}},q)} \textnormal{ and } T_{\textnormal{diag}(1,\underbrace{p,\dots,p}_{n-2 \textnormal{ terms}},p^2)}\circ T_{\textnormal{diag}(1,\underbrace{q,\dots,q}_{n-2 \textnormal{ terms}},q^2)}
\end{multline*}
where $p$ and $q$ are two prime numbers. The results are given in the next section and reveal that the relevant Hecke operators when applying the amplification method in $GL(n)$ are
\begin{equation*}
T_{\textnormal{diag}(1,\underbrace{p,\dots,p}_{n-2 \textnormal{ terms}},pq)}, T_{\textnormal{diag}(1,\underbrace{pq,\dots,pq}_{n-2 \textnormal{ terms}},(pq)^2)}, T_{\textnormal{diag}(1,\underbrace{p,\dots,p}_{n-2 \textnormal{ terms}},p^2)}
\end{equation*}
and
\begin{equation*}
T_{\textnormal{diag}(1,\underbrace{p^2,\dots,p^2}_{n-3 \textnormal{ terms}},p^3,p^3)}, T_{\textnormal{diag}(1,1,\underbrace{p,\dots,p}_{n-3 \textnormal{ terms}},p^3)}, T_{\textnormal{diag}(1,1,\underbrace{p,\dots,p}_{n-4 \textnormal{ terms}},p^2,p^2)}.
\end{equation*}
%...............................................................................................
\subsection{Statement of the results}%
\begin{theoint}\label{coro_A}
Let $n\geq 4$, $\Lambda_n=GL_n(\Z)$ and $p$ be a prime number.
\begin{itemize}
\item
The finite set $R^{(n)}(p)$ of cardinality
\begin{equation*}
\textnormal{deg}\left(\textnormal{diag}\left(1,\underbrace{p,\dots,p}_{n-2 \textnormal{ terms}},p^2\right)\right)=p\frac{\left(p^{n-1}-1\right)\left(p^n-1\right)}{(p-1)^2}
\end{equation*}
defined in Proposition \ref{propo_representatives} is a complete system of representatives of the distinct $\Lambda_n$ right cosets of
\begin{equation*}
\Lambda_n\textnormal{diag}\left(1,\underbrace{p,\dots,p}_{n-2 \textnormal{ terms}},p^2\right)\Lambda_n
\end{equation*}
modulo $\Lambda_n$.
\item
The following formulas for the degrees hold.
\begin{eqnarray*}
\textnormal{deg}\left(\textnormal{diag}\left(p,\underbrace{p^2,\dots,p^2}_{n-2 \textnormal{ terms}},p^3\right)\right) & = & p\frac{\left(p^{n-1}-1\right)\left(p^n-1\right)}{(p-1)^2}, \\
\textnormal{deg}\left(\textnormal{diag}\left(1,\underbrace{p^2,\dots,p^2}_{n-3 \textnormal{ terms}},p^3,p^3\right)\right) & = & p^{n+1}\frac{\left(p^{n-2}-1\right)\left(p^{n-1}-1\right)\left(p^n-1\right)}{(p-1)^2(p^2-1)}
\end{eqnarray*}
and
\begin{eqnarray*}
\textnormal{deg}\left(\textnormal{diag}\left(1,\underbrace{p^2,\dots,p^2}_{n-2 \textnormal{ terms}},p^4\right)\right) & = & p^{2n-1}\frac{\left(p^{n-1}-1\right)\left(p^n-1\right)}{(p-1)^2}, \\
\textnormal{deg}\left(\textnormal{diag}\left(p,p,\underbrace{p^2,\dots,p^2}_{n-3 \textnormal{ terms}},p^4\right)\right) & = & p^{n+1}\frac{\left(p^{n-2}-1\right)\left(p^{n-1}-1\right)\left(p^n-1\right)}{(p-1)^2(p^2-1)}, \\
\textnormal{deg}\left(\textnormal{diag}\left(p,p,\underbrace{p^2,\dots,p^2}_{n-4 \textnormal{ terms}},p^3,p^3\right)\right) & = & p^4\frac{\left(p^{n-3}-1\right)\left(p^{n-2}-1\right)\left(p^{n-1}-1\right)\left(p^{n}-1\right)}{(p-1)^2\left(p^2-1\right)^2}.
\end{eqnarray*}
\item
Finally,
\begin{multline}\label{eq_coro_A_2}
\Lambda_n\textnormal{diag}\left(1,\underbrace{p,\dots,p}_{n-2 \textnormal{ terms}},p^2\right)\Lambda_n\ast \Lambda_n\textnormal{diag}\left(1,\underbrace{p,\dots,p}_{n-2 \textnormal{ terms}},p^2\right)\Lambda_n \\
=\frac{2p^n-p^2-2p+1}{p-1}\Lambda_n\textnormal{diag}\left(p,\underbrace{p^2,\dots,p^2}_{n-2 \textnormal{ terms}},p^3\right)\Lambda_n \\
+p\frac{\left(p^{n-1}-1\right)\left(p^n-1\right)}{(p-1)^2}\Lambda_n\textnormal{diag}\left(\underbrace{p^2,\dots,p^2}_{n \textnormal{ terms}}\right)\Lambda_n+\Lambda_n\textnormal{diag}\left(1,\underbrace{p^2,\dots,p^2}_{n-2 \textnormal{ terms}},p^4\right)\Lambda_n \\
+(p+1)\Lambda_n\textnormal{diag}\left(1,\underbrace{p^2,\dots,p^2}_{n-3 \textnormal{ terms}},p^3,p^3\right)\Lambda_n+(p+1)\Lambda_n\textnormal{diag}\left(p,p,\underbrace{p^2,\dots,p^2}_{n-3 \textnormal{ terms}},p^4\right)\Lambda_n \\
+(p+1)^2\Lambda_n\textnormal{diag}\left(p,p,\underbrace{p^2,\dots,p^2}_{n-4 \textnormal{ terms}},p^3,p^3\right)\Lambda_n.
\end{multline}
\end{itemize}
\end{theoint}
\begin{corint}\label{theo_A}
Let $n\geq 4$. If $p$ and $q$ are two prime numbers then
\begin{equation}\label{eq_lin_1}
T_{\textnormal{diag}(1,\underbrace{p,\dots,p}_{n-1 \textnormal{ terms}})}\circ T_{\textnormal{diag}(\underbrace{1,\dots,1}_{n-1 \textnormal{ terms}},q)}=T_{\textnormal{diag}(1,\underbrace{p,\dots,p}_{n-2 \textnormal{ terms}},pq)}+\delta_{p=q}\frac{p^{n}-1}{p-1}\textnormal{Id}
\end{equation}
and
\begin{multline}\label{eq_lin_2}
T_{\textnormal{diag}(1,\underbrace{p,\dots,p}_{n-2 \textnormal{ terms}},p^2)}\circ T_{\textnormal{diag}(1,\underbrace{q,\dots,q}_{n-2 \textnormal{ terms}},q^2)}=T_{\textnormal{diag}(1,\underbrace{pq,\dots,pq}_{n-2 \textnormal{ terms}},(pq)^2)} \\
+\delta_{p=q}\frac{2p^n-p^2-2p+1}{p-1}T_{\textnormal{diag}(1,\underbrace{p,\dots,p}_{n-2 \textnormal{ terms}},p^2)}+\delta_{p=q}p\frac{\left(p^{n-1}-1\right)\left(p^n-1\right)}{(p-1)^2}\textnormal{Id} \\
+\delta_{p=q}(p+1)T_{\textnormal{diag}(1,\underbrace{p^2,\dots,p^2}_{n-3 \textnormal{ terms}},p^3,p^3)}+\delta_{p=q}(p+1)T_{\textnormal{diag}(1,1,\underbrace{p,\dots,p}_{n-3 \textnormal{ terms}},p^3)} \\
+\delta_{p=q}(p+1)^2T_{\textnormal{diag}(1,1,\underbrace{p,\dots,p}_{n-4 \textnormal{ terms}},p^2,p^2)}.
\end{multline}
\end{corint}
When $p\neq q$, the previous corollary follows from \eqref{eq_milti_33} whereas when $p=q$, it comes from Theorem \ref{coro_A}, \cite[Lemma 2.18 Page 114]{MR1349824} and \eqref{eq_ex_multi}. This corollary generalizes the case $n=2$, well-known for a long time, and the case $n=3$ done in \cite{HoRiRo2}.
%...............................................................................................
\subsection{On the possible applications of this higher rank amplifier}%
%................................................................................................
\subsubsection{Subconvexity bounds for $L$-functions}%
Let $f$ be a $GL(n)$ Hecke Maa$\ss$ cusp form. A very classical problem considered by analytic number theorists is the size of the Godement-Jacquet $L$-function associated to $f$, say $L(f,s)$ with $s$ on the critical line $\Re{(s)}=1/2$ when the analytic conductor $C(f)$ of $f$ tends to infinity. The bound
\begin{equation*}
L(f,s)\ll C(f)^{1/4},
\end{equation*}  
is named the convexity or trivial bound, even if this is not a trivial result in general. Improving this bound, namely proving a subconvexity bound, was proved in the past to be useful to solve many arithmetical questions, such as equidistribution results.
\par
The $GL(2)$ case was intensively investigated in the last decades, culminating in the work of P.~Michel and A.~Venkatesh in \cite{MR2653249}, who used the amplification method in $GL(2)$. It seems that the best subconvexity bounds in the $GL(2)$ case intrinsic to the amplification method are the Weyl exponent $1/4(1-1/3)$ (\cite{Wey}) and the Burgess exponent $1/4(1-1/4)$ (\cite{MR0132733}).
\par
Very few examples of subconvexity bounds for $L$-functions of $GL(n)$ automorphic forms, which are not lifts of $GL(2)$ ones, are known. One can quote \cite{MR2753605}, \cite{MR2975240}, \cite{Mu1} in the rank $2$ case and an extremely recent and elaborate subconvexity bound for twisted $L$-functions of $GL(3)$ automorphic forms by R.~Munshi in \cite{Mu2}. As far as we know, the Weyl and Burgess exponents have never appeared in this higher rank case.
\par
We hope that the completely explicit $GL(n)$ amplifier built in this paper will sheld some new lights on these questions in the close future.
%.............................................................................................
\subsubsection{Subconvexity bounds for sup-norms of automorphic forms}%
Let $f$ be a $L^2$-normalized $GL(n)$ Hecke Maa$\ss$ cusp form.
\par
\noindent\textit{{The spectral aspect.}}
\par
Let $K$ be a fixed compact subset of $SL_n(\R)/SO_n(\R)$. The convexity bound for the sup-norm of $f$ restricted to $K$ is given by
\begin{equation*}
\abs{\abs{f\vert_{K}}}_\infty\ll\lambda_f^{n(n-1)/8}
\end{equation*}
where $\lambda_f$ is the Laplace eigenvalue of $f$. It is important to mention that F.~Brumley and N.~Templier discovered in \cite{BrTe} that this convexity bound does not hold when $n\geq 6$ if $f$ is not restricted to a compact.
\par
The convexity bound is not expected to be sharp, essentially because they are some additional symmetries on $SL_n(\R)/SO_n(\R)$: the Hecke correspondences. More precisely, one should be able to prove a subconvexity bound, namely finding an absolute positive constant $\delta_n>0$ such that
\begin{equation}\label{eq_expect}
\abs{\abs{f\vert_{K}}}_\infty\ll\lambda_f^{n(n-1)/8-\delta_n}
\end{equation}
\par
The pioneering work done by H.~Iwaniec and P.~Sarnak in \cite{MR1324136} is the bound given in \eqref{eq_expect} when $n=2$ for $\delta_2=1/24$. This constant $\delta_2$ seems to be intrinsic to the amplification method in $GL(2)$. The case $n=3$ was completed in \cite{HoRiRo}. The general case was done in a series of impressive works by V.~Blomer and P.~Maga in \cite{BlMa1} and in \cite{BlMa2}. One could also quote \cite{Ma3}.
\par
All these achievements were done thanks to the amplification method. Determining what should be the best subconvexity exponent intrinsic to the amplification method is an interesting question, which should reveal new types of analytic problems. Needless to say that the explicit $GL(n)$ amplifier could be useful to do so.
\par
\noindent\textit{{The level aspect.}}
\par
Let say that $f$ is of level $q$ and let us speak about the growth of the sup-norm of $f$ as $q$ gets large.
\par
For $GL(2)$ and when the level $q$ is squarefree, the convexity bound is
\begin{equation*}
\abs{\abs{f}}_\infty\ll q^\epsilon
\end{equation*}
for all $\epsilon>0$ but one expects that the correct order of magnitude is
\begin{equation*}
\abs{\abs{f}}_\infty\ll q^{-1/2+\epsilon}
\end{equation*}
This rank $1$ case in prime level was intensively studied during the last years after the foundational work of V.~Blomer and R.~Holowinsky in \cite{MR2587342}. See \cite{MR2734340}, \cite{MR2989618} and \cite{HeRi}. In \cite{MR3038127}, the authors proved the bound
\begin{equation*}
\abs{\abs{f}}_\infty\ll q^{-1/6+\epsilon}
\end{equation*}
which seems to be the best possible subconvexity exponent intrinsic to the amplification method. Note that the authors really used the shape of the explicit $GL(2)$ amplifier in order to get this bound. When the level $q$ is not squarefree, the situation is more delicate since the Atkin-Lehner group has more than one orbit when acting on the cusps. See \cite{Sa} and \cite{Ma4} for more details.
\par
For $GL(n)$, as far as we know, these questions remain completely open. We hope that the explicit $GL(n)$ amplifier constructed in this work will make possible an investigation of these questions in a higher rank setting.
%....................................................................................
\subsection{Organization of the paper}%
The general background on $GL(n)$ Maa$\ss$ cusp forms and on the $GL(n)$ Hecke algebra is given in Section \ref{sec_GL3}. The proof of the first bullet in Theorem \ref{coro_A} is done in Section \ref{sec_dec} (see Proposition \ref{propo_representatives}). The proof of both the formulas for the degrees given in Theorem \ref{coro_A} and equation \eqref{eq_coro_A_2} is detailed in Section \ref{sec_proof_eq}.  
%........................................................................................
\begin{notations}
$n\geq 2$ is an integer and $p, q$ are prime numbers. $\Lambda_n$ stands for the group $GL_n(\Z)$ of invertible matrices of size $n$ with integer coefficients, whose unity element is the identity matrix $I_n$. If $a_1,\dots,a_n$ are real numbers then $\textnormal{diag}(a_1,\dots,a_n)$ denotes the diagonal matrix of size $n$ with $a_1,\dots,a_n$ as diagonal coefficients. The following double $\Lambda_n$ cosets will occur throughout this article:
\begin{eqnarray*}
\pi_i^{(n)}(p) & \coloneqq & \Lambda_n D_i^{(n)}(p)\Lambda_n,\quad D_i^{(n)}(p)=\textnormal{diag}\left(1,\dots,1,\underbrace{p,\dots,p}_{i \textnormal{terms}}\right), \\
\pi^{(n)}(p) & \coloneqq & \Lambda_n D^{(n)}(p)\Lambda_n,\quad D^{(n)}(p)=\textnormal{diag}\left(1,\underbrace{p,\dots,p}_{n-2 \textnormal{ terms}},p^2\right), \\
\pi_{i,j}^{(n)}(p)& \coloneqq & \Lambda_n D_{i,j}^{(n)}(p)\Lambda_n,\quad D_{i,j}^{(n)}(p)=\textnormal{diag}\left(1,\dots,1,\underbrace{p,\dots,p}_{i \text{terms}},\underbrace{p^2,\dots,p^2}_{j \text{ terms}}\right)
\end{eqnarray*}
for $0\leq i,j\leq n$ with $i+j\leq n$. The following polynomials in $x$ will occur when computing the degrees of some relevant $\Lambda_n$ double cosets for this work:
\begin{equation*}
\varphi_r(x)\coloneqq\prod_{k=1}^r\left(x^k-1\right),\quad \varphi_0(x)=1
\end{equation*}
for $r\geq 1$. Let us denote by $\uple{d}_n(p)$ the $n$-uple $\left(1,p,\dots,\underbrace{p^{k-1}}_{\text{$k$'th term}},\dots,p^{n-2},p^n\right)$ for $2\leq k\leq n-1$. Finally, if $\mathcal{P}$ is a property then $\delta_{\mathcal{P}}$ is the Kronecker symbol, namely $1$ if $\mathcal{P}$ is satisfied and $0$ otherwise.
\end{notations}
%................................................................................
\begin{merci}%
The author would like to thank R.~Holowinsky and E.~Royer for fruitful discussions related to this work.
\par
The author's research was partially supported by a Marie Curie Intra European Fellowship within the 7th European Community Framework Programme. The grant agreement number of this project, whose acronym is ANERAUTOHI, is PIEF-GA-2009-25271. He would like to thank E.~Kowalski, ETH and its entire staff for the excellent working conditions.
\par
Last but not least, the author would like to express his gratitude to K.~Belabas for his crucial but isolated support for Analytic Number Theory among the Number Theory research team A2X (Institut de Math\'ematiques de Bordeaux, Universit\'e de Bordeaux).
\end{merci}
%...............................................................................................
\section{Background on the $GL(n)$ Hecke algebra}\label{sec_GL3}%
In this section, $n\geq 2$. The convenient references for this section are \cite{MR1349824}, \cite{MR2254662}, \cite{MR1027069}, \cite{MR0340283} and \cite{MR1291394}.
\par
Let $f$ be a $GL(n)$ Maa$\ss$ cusp form of level $1$. Such $f$ admits a Fourier expansion
\begin{multline}\label{eq_fourier}
f(g)=\sum_{\gamma\in U_{n-1}(\Z)\setminus SL_{n-1}(\Z)}\sum_{\substack{m_1,\dots,m_{n-2}\geq 1 \\
m_{n-1}\in\Z^\ast}}\frac{a_f\left(m_1,\dots,m_{n-1}\right)}{\prod_{1\leq k\leq n-1}\abs{m_k}^{k(n-k)/2}} \\
W_{\text{Ja}}\left(\text{diag}\left(m_1\dots m_{n-2}\abs{m_{n-1}},\dots,m_1m_2,m_1,1\right)\begin{pmatrix}
\gamma & \\
& 1
\end{pmatrix}g,\nu_f,\psi_{\underbrace{1,\dots,1}_{n-2 \text{ terms}},\frac{m_{n-1}}{\abs{m_{n-1}}}}\right)
\end{multline}
for $g\in GL_n(\R)$ (see \cite[Equation (9.1.2)]{MR2254662}. Here $U_{n-1}(\Z)$ stands for the $\Z$-points of the group of upper-triangular unipotent matrices of size $n-1$. $\nu_f\in\C^{n-1}$ is the type of $f$, whose components are complex numbers characterized by the property that, for every invariant differential operator $D$ in the center of the universal enveloping algebra of $GL_n(\R)$, the cusp form $f$ is an eigenfunction of $D$ with the same eigenvalue as the power function $I_{\nu_f}$, which is defined in \cite[Equation (5.1.1)]{MR2254662}. $\psi_{\underbrace{1,\dots,1}_{n-2 \text{ terms}},\pm 1}$ is the character of the group of upper-triangular unipotent real matrices of size $n$ defined by
\begin{equation*}
\psi_{\underbrace{1,\dots,1}_{n-2 \text{ terms}},\pm 1}(u)=e^{2i\pi\left(u_{1,2}+\dots+u_{n-2,n-1}\pm u_{n-1,n}\right)}.
\end{equation*}
for $u=\left[u_{i,j}\right]_{1\leq i,j\leq n}$. $W_{Ja}\left(\ast,\nu_f,\psi_{\underbrace{1,\dots,1}_{n-2 \text{ terms}},\pm 1}\right)$ stands for the $GL(n)$ Jacquet Whittaker function of type $\nu_f$ and character $\psi_{\underbrace{1,\dots,1}_{n-2 \text{ terms}},\pm 1}$ defined in \cite[Equation 6.1.2]{MR2254662}. The complex number $a_f(m_1,\dots,m_{n-1})$ is the $(m_1,\dots,m_{n-1})$'th \emph{Fourier coefficient} of $f$ for $m_1,\dots,m_{n-2}$ some positive integers and $m_{n-1}$ a non-vanishing integer.
\par
For $g\in GL_n(\Q)$, one knows (see \cite[Lemma 1.2 Page 94 and Lemma 2.1 Page 105]{MR1349824}) that the $\Lambda_n$ double coset $\Lambda_ng\Lambda_n$ is a finite union of $\Lambda_n$ right cosets such that it makes sense to define the \emph{Hecke operator} $T_g$ by
\begin{equation*}
T_g(f)(h)=\sum_{\delta\in\Lambda_n\setminus\Lambda_n g\Lambda_n}f(\delta h)
\end{equation*}
for $h\in GL_n(\R)$ (see \cite[Chapter 3, Sections 1.1 and 1.5]{MR1349824}. The \emph{degree} of $g$ or $T_g$ is defined by
\begin{equation*}
\text{deg}(g)=\text{deg}(T_g)=\text{card}\left(\Lambda_n\setminus\Lambda_n g\Lambda_n\right).
\end{equation*}
Obviously,
\begin{equation}\label{eq_dilate}
\text{deg}(rg)=\text{deg}(g).
\end{equation}
for $r\in\Q^\times$. By \cite[Lemma 2.18 Page 114]{MR1349824},
\begin{equation}\label{eq_main_degree}
\text{deg}\left(D_{i,j}^{(n)}(p)\right)=p^{j(n-i-j)}\frac{\varphi_n(p)}{\varphi_{n-i-j}(p)\varphi_{i}(p)\varphi_{j}(p)}
\end{equation}
for $0\leq i,j\leq n$ with $i+j\leq n$. The adjoint of $T_g$ for the Peterson inner product is $T_{g^{-1}}$. The algebra of Hecke operators $\mathbb{T}$ is the ring of endomorphisms generated by all the $T_g$'s with $g\in GL_n(\Q)$, a commutative algebra of normal endomorphisms (see \cite[Theorem 9.3.6]{MR2254662}), which contains the $m$'th normalised Hecke operator
\begin{equation*}
T_m=\frac{1}{m^{(n-1)/2}}\sum_{\substack{g=\text{diag}(y_1,y_2,y_3) \\
y_1\mid y_2\mid\dots\mid y_n\\
y_1y_2\dots y_n=m}}T_g
\end{equation*}
for all positive integer $m$. A \emph{Hecke-Maa$\ss$ cusp form} $f$ of level $1$ is a Maa$\ss$ cusp form of level $1$, which is an eigenfunction of $\mathbb{T}$. In particular, it satisfies
\begin{equation}\label{eq_Hecke}
T_m(f)=a_f(m,\underbrace{1,\dots,1}_{n-2 \text{ terms}})f \text{ and } T_m^\ast(f)=a_f(\underbrace{1,\dots,1}_{n-2 \text{ terms}},m)f
\end{equation}
according to \cite[Theorem 9.3.11]{MR2254662}.
\par
The algebra $\mathbb{T}$ is isomorphic to the \emph{absolute Hecke algebra}, the free $\mathbb{Z}$-module generated by the double cosets $\Lambda_ng\Lambda_n$ where $g$ ranges over $\Lambda_n\setminus GL_n(\Q)/\Lambda_n$ and endowed with the following multiplication law. If $g_1$ and $g_2$ belong to $GL_n(\Q)$ and
\begin{equation*}
\Lambda_ng_1\Lambda_n=\bigcup_{i=1}^{\text{deg}(g_1)}\Lambda_n\alpha_i \text{ and } \Lambda_ng_2\Lambda_n=\bigcup_{j=1}^{\text{deg}(g_2)}\Lambda_n\beta_j 
\end{equation*}
then
\begin{equation}\label{eq_product}
\Lambda_ng_1\Lambda_n\ast\Lambda_ng_2\Lambda_n=\sum_{\Lambda_nh\Lambda_n\subset\Lambda_ng_1\Lambda_ng_2\Lambda_n}m(g_1,g_2;h)\Lambda_nh\Lambda_n
\end{equation}
where $h\in GL_n(\Q)$ ranges over a system of representatives of the $\Lambda_n$-double cosets contained in the set $\Lambda_ng_1\Lambda_ng_2\Lambda_n$ and
\begin{align}\label{eq_multi}
m(g_1,g_2;h) & =\text{card}\left(\left\{(i,j)\in\{1,\dots,\text{deg}(g_1)\}\times\{1,\dots,\text{deg}(g_2)\}, \alpha_i\beta_j\in\Lambda_nh\right\}\right), \\
& =\frac{1}{\text{deg}(h)}\text{card}\left(\left\{(i,j)\in\{1,\dots,\text{deg}(g_1)\}\times\{1,\dots,\text{deg}(g_2)\}, \alpha_i\beta_j\in\Lambda_nh\Lambda_n\right\}\right), \\
& =\frac{\text{deg}(g_2)}{\text{deg}(h)}\text{card}\left(\left\{i\in\{1,\dots,\text{deg}(g_1)\}, \alpha_ig_2\in\Lambda_nh\Lambda_n\right\}\right).
\end{align}
Confer \cite[Lemma 1.5 Page 96]{MR1349824}. In particular,
\begin{equation}\label{eq_ex_multi}
\Lambda_nrI_n\Lambda_n\ast\Lambda_n g\Lambda_n=\Lambda_n rg\Lambda_n
\end{equation}
for $g\in GL_n(\Q)$ and $r\in\Q^\times$ (\cite[Lemma 2.4 Page 107]{MR1349824}).
\par
For $g\in GL_n(\Q)$ with integer coefficients, the $\Lambda_n$ right coset $\Lambda_ng$ contains a unique upper-triangular column reduced matrix, namely
\begin{equation}\label{eq_Hermite}
\Lambda_ng=\Lambda_nC
\end{equation}
where $C=\left[c_{i,j}\right]_{1\leq i,j\leq n}$ is an upper-triangular matrix with integer coefficients satisfying
\begin{equation*}
\forall j\in\{2,\dots,n\},\forall i\in\{1,j-1\},\quad 0\leq c_{i,j}<c_{j,j}
\end{equation*}
by \cite[Lemma 2.7]{MR1349824}.
\par
Let $g$ be a matrix of size $n$ with integer coefficients. Let $1\leq k\leq n$. Let $I_{n,k}$ be the set of all k-tuples $\{i_1,\dots,i_n\}$ satisfying $1\leq i_1<i_2<\dots<i_k\leq n$. Obviously, $I_{n,k}$ is of cardinal $\binom{n}{k}$. If $\omega$ and $\tau$ are two elements of $I_{n,k}$ then $g(\omega,\tau)$ will denote the $k\times k$ determinantal minor of $g$ whose row indices are the elements of $\omega$ and whose column indices are the elements of $\tau$. Obviously, there are $\binom{n}{k}^2$ such minors. The $k$'th \emph{determinantal divisor} of $g$, say $d_k(g)$, is the non-negative integer defined by
\begin{equation}\label{eq_def_dk}
d_k(g)=\begin{cases}
0 & \text{if $\forall(\omega,\tau)\in I_{n,k}^2, g(\omega,\tau)=0$,} \\
gcd_{(\omega,\tau)\in I_{n,k}^2}g(\omega,\tau) & \text{otherwise}
\end{cases}
\end{equation}
and the \emph{determinantal vector} of $g$ is $\uple{d}_n(g)=\left(d_1(g),\dots,d_n(g)\right)$. The determinantal divisors turn out to be usefull since if $h$ is another matrix of size $n$ with integer coefficients then
\begin{equation}
h\in\Lambda_ng\Lambda_n \text{ if and only if } \uple{d}(h)=\uple{d}(g)
\end{equation}
according to \cite{MR0340283}.
\par
By \cite[Proposition 2.5 Page 107]{MR1349824}, if $g_1,g_2$ belong to $GL_n(\Q)$ with integer coefficients then
\begin{equation}\label{eq_milti_33}
\Lambda_ng_1\Lambda_n\ast\Lambda_ng_2\Lambda_n=\Lambda_ng_1g_2\Lambda_n
\end{equation}
provided $d_1(g_1)=d_1(g_2)=1$ and $(d_n(g_1),d_n(g_2))=1$.
\par
Finally, we will use the following result on the local integral Hecke algebra at the prime $p$, say $\underline{H}_p^n$, defined as the $\Lambda_n$ double cosets $\Lambda_ng\Lambda_n$, where $g$ ranges over the matrices in $GL_n\left(\Z[1/p]\right)$ with integer coefficients. By \cite[Lemma 2.16 Page 112]{MR1349824}, the $\Q$-linear map $\Psi:\underline{H}_p^{n}\to\underline{H}_p^{n-1}$ defined by
\begin{equation}\label{eq_psi}
\Psi\left(\Lambda_n\text{diag}\left(p^{\delta_1},\dots,p^{\delta_n}\right)\Lambda_n\right)=\begin{cases}
\Lambda_n\text{diag}\left(p^{\delta_2},\dots,p^{\delta_n}\right)\Lambda_n & \text{if $0=\delta_1\leq\delta_2\leq\dots\leq\delta_n$,} \\
0 & \text{otherwise}
\end{cases}
\end{equation}
is a morphism of rings.
%.................................................................................................
\section{Decomposition of $\pi^{(n)}(p)$ into $\Lambda_n$ right cosets}\label{sec_dec}%
In this section, $n\geq 2$. The main purpose of this section is to find a convenient complete system of representatives for the distinct $\Lambda_n$ right cosets of $\pi^{(n)}(p)$ modulo $\Lambda_n$. Let us denote by $R_0^{(n)}(p)$ the set of upper-triangular matrices $C=\left[c_{i,j}\right]_{1\leq i,j\leq n}$ of size $n$ with integer coefficients satisfying
\begin{equation}\label{eq_0_1}
\uple{d}_n(C)=\uple{d}_n(p),
\end{equation}
\begin{equation}\label{eq_0_2}
\forall i\in\left\{1,\dots,n\right\},\quad c_{i,i}=p,
\end{equation}
and
\begin{equation}\label{eq_0_3}
\forall j\in\left\{2,\dots,n\right\}, \forall i\in\left\{1,\dots,j-1\right\},\quad 0\leq c_{i,j}<p.
\end{equation}
Let us also denote by $R_1^{(n)}(p)$ the set of upper-triangular matrices $C=\left[c_{i,j}\right]_{1\leq i,j\leq n}$ of size $n$ with integer coefficients satisfying
\begin{equation}\label{eq_1_1}
\forall i\in\left\{1,\dots,n\right\},\quad c_{i,i}\in\left\{1,p,p^2\right\}, 
\end{equation}
\begin{equation}\label{eq_1_2}
\exists!i\in\left\{1,\dots,n\right\},\quad c_{i,i}=1\quad\text{ and }\quad\exists!i\in\left\{1,\dots,n\right\},\quad c_{i,i}=p^2,
\end{equation}
\begin{equation}\label{eq_1_3}
\forall j\in\left\{2,\dots,n\right\}, \forall i\in\left\{1,\dots,j-1\right\},\quad 0\leq c_{i,j}<c_{j,j}
\end{equation}
and
\begin{equation}\label{eq_1_4}
\forall i\in\left\{1,\dots,n-1\right\}, p\mid c_{i,i}\Rightarrow\forall j\in\left\{i+1,\dots,n\right\},\quad p\mid c_{i,j}.
\end{equation}
\begin{proposition}\label{propo_representatives}
Let $n\geq 2$. The set $R^{(n)}(p)=R_0^{(n)}(p)\sqcup R_1^{(n)}(p)$ is a complete system of representatives of the distinct $\Lambda_n$ right cosets of $\pi^{(n)}(p)$ modulo $\Lambda_n$. In other words,
\begin{equation*}
\pi^{(n)}(p)=\left(\bigsqcup_{C_0\in R_0^{(n)}(p)}\Lambda_nC_0\right)\bigsqcup\left(\bigsqcup_{C_1\in R_1^{(n)}(p)}\Lambda_nC_1\right).
\end{equation*}
In addition,
\begin{eqnarray*}
\text{card}\left(R_0^{(n)}(p)\right) & = & \frac{(n-1)p^{n}-np^{n-1}+1}{p-1}, \\
\text{card}\left(R_1^{(n)}(p)\right) & = & \frac{p^{2n}-np^{n+1}+2(n-1)p^n-np^{n-1}+1}{(p-1)^2}.
\end{eqnarray*}
\end{proposition}
\begin{proof}[\proofname{} of Proposition \ref{propo_representatives}]%
By \eqref{eq_1_4}, all the matrices $C_1$ in $R_1^{(n)}(p)$ can be decomposed into
\begin{equation*}
C_1=\text{diag}\left(p^{\alpha_1},\dots,p^{\alpha_n}\right)C_1^\prime
\end{equation*}
for some non negative integers $\alpha_1,\dots,\alpha_n$ and with $C_1^\prime\in\Lambda_n$ such that
\begin{equation*}
C_1\in\Lambda\text{diag}\left(p^{\alpha_1},\dots,p^{\alpha_n}\right)\Lambda=\pi^{(n)}(p)
\end{equation*}
by \eqref{eq_1_1} and \eqref{eq_1_2}.
\par
All the matrices $C_0$ in $R_0^{(n)}(p)$ belong to $\pi^{(n)}(p)$ since their determinantal vectors match the determinantal vector of $D^{(n)}(p)$ by \eqref{eq_0_1}.
\par
All the matrices in $R^{(n)}(p)$ are upper-triangular column reduced matrices by \eqref{eq_0_3}, \eqref{eq_1_3} and belong to different $\Lambda_n$ right cosets according to the unicity statement given in \eqref{eq_Hermite}.
\par
Let $C=\left[c_{i,j}\right]_{1\leq i,j\leq n}$ be any upper-triangular column reduced matrix that lies in $\pi^{(n)}(p)$ and let us prove that $C$ belongs to $R^{(n)}(p)$. First of all, the determinant of $C$ is $p^n$ such that
\begin{equation*}
\forall i\in\left\{1,\dots,n\right\},\exists\alpha_i\in\mathbb{N},\quad c_{i,i}=p^{\alpha_i}.
\end{equation*}
Then, $C=\lambda_1D^{(n)}(p)\lambda_2$ with $\lambda_1,\lambda_2$ in $\Lambda_n$, which entails that $C^{-1}=\lambda_2^{-1}D^{(n)}(p)^{-1}\lambda_1^{-1}$. As a consequence, $p^2C^{-1}$ has integer coefficients and
\begin{equation*}
\forall i\in\left\{1,\dots,n\right\},\quad \alpha_i\in\{0,1,2\}.
\end{equation*}
If all the diagonal coefficients of $C$ are equal to $p$ then $C$ belongs to $R_0^{(n)}(p)$ since its determinantal vector must me equal to the determinantal vector of $D^{(n)}(p)$, namely $\uple{d}_n(p)$. Assume that one of its diagonal coefficient is not equal to $p$. The condition $d_2(C)=p$ implies that there must be at most one diagonal coefficient of $C$ equal to $1$. Let us prove that $C$ has a single diagonal coefficient equal to $1$ and a single coefficient equal to $p^2$. Let $\sigma$ be the permutation of $\{1,\dots,n\}$ satisfying
\begin{equation*}
0\leq\alpha_{\sigma(1)}\leq\dots\leq\alpha_{\sigma(n)}\leq 2.
\end{equation*}
The determinant condition is
\begin{equation*}
\alpha_{\sigma(1)}+\dots+\alpha_{\sigma(n)}=n.
\end{equation*}
If $\alpha_{\sigma(1)}=0$ then $\alpha_{\sigma(2)}=\dots=\alpha_{\sigma(n-1)}=1$ and $\alpha_{\sigma(n)}=2$ by Lemma \ref{lemma_induc_trivial}. If $\alpha_{\sigma(1)}\geq 1$ then all the diagonal coefficients of $C$ are equal to $p$ by Lemma \ref{lemma_induc_trivial}, which is a contradiction. Thus, \eqref{eq_1_2} is satisfied. Let us prove \eqref{eq_1_4}. Assume on the contrary that there exist $i_0$ in $\{1,\dots,n-1\}$ and $j_0$ in $\{i_0+1,\dots,n\}$ such that $p\mid c_{i_0,i_0}$ and $p\nmid c_{i_0,j_0}$. The fact that $p\nmid c_{i_0,j_0}$ implies that $c_{j_0,j_0}\neq 1$. Let $j_1\neq j_0$ be the index of the column of $C$, for which $c_{j_1,j_1}=1$. Let us prove that the columns $C[j_1]$ of $C$ of index $j_1$ and $C[j_0]$ of $C$ of index $j_0$ are linearly independent modulo $p$. If
\begin{equation*}
0=\lambda_0C[j_0]+\lambda_1C[j_1]
\end{equation*}
then the $i_0$'th component implies that
\begin{equation*}
0=\lambda_0c_{i_0,j_0}+\lambda_1c_{i_0,j_1}=\lambda_0c_{i_0,j_0}
\end{equation*}
such that $\lambda_0=0$ since $c_{i_0,j_0}$ is invertible modulo $p$ and $\lambda_1=0$. This is a contradiction since $C$ is of rank $1$ modulo $p$. Thus, $C$ belongs to $R_1^{(n)}(p)$.
\par
Let us compute the cardinality of $R_1^{(n)}(p)$. Obviously,
\begin{equation*}
\text{card}\left(R_1^{(n)}(p)\right)=p^{n-1}\sum_{1\leq\alpha_1<\alpha_2\leq n}p^{\alpha_2-\alpha_1}+p^{n-1}\sum_{1\leq\alpha_2<\alpha_1\leq n}p^{\alpha_2-\alpha_1}.
\end{equation*}
A straightforward computation ensures that
\begin{equation*}
\sum_{1\leq a<b\leq n}x^{a-b}=\frac{(n-1)x^n-nx^{n-1}+1}{x^{n-1}(x-1)^2}
\end{equation*}
for all real number $x\neq 1$, from which the computation of this cardinality follows, with $x=p$ in the first sum and $x=1/p$ in the second one.
\par
Let us compute the cardinality of $R_0^{(n)}(p)$. Obviously,
\begin{align*}
\text{card}\left(R_0^{(n)}(p)\right) & =\text{card}\left(R^{(n)}(p)\right)-\text{card}\left(R_1^{(n)}(p)\right) \\
& =\text{deg}\left(D^{(n)}(p)\right)-\text{card}\left(R_1^{(n)}(p)\right) \\
& =p\frac{\varphi_n(p)}{\varphi_1(p)^2\varphi_{n-2}(p)}-\frac{p^{2n}-np^{n+1}+2(n-1)p^n-np^{n-1}+1}{(p-1)^2} \\
& =p\frac{\left(p^{n-1}-1\right)\left(p^n-1\right)}{(p-1)^2}-\frac{p^{2n}-np^{n+1}+2(n-1)p^n-np^{n-1}+1}{(p-1)^2}
\end{align*}
by \eqref{eq_main_degree}.
\end{proof}
The following lemma, which follows from a simple induction, has been used in the previous proof.
\begin{lemma}\label{lemma_induc_trivial}
Let $n\geq 4$ be an integer and $\alpha_1,\dots,\alpha_n$ be non negative integers.
\begin{itemize}
\item
If $1\leq\alpha_1\leq\dots\leq\alpha_n\leq 2$ and $\alpha_1+\dots+\alpha_n=n$ then $\alpha_1=\dots=\alpha_n=1$.
\item
If $1\leq\alpha_2\leq\dots\leq\alpha_n\leq 2$ and $\alpha_2+\dots+\alpha_n=n$ then $\alpha_2=\dots=\alpha_{n-1}=1$ and $\alpha_n=2$.
\end{itemize}
\end{lemma}
We will need more details, stated in the following proposition, on the matrices in $R_0^{(n)}(p)$.
\begin{proposition}\label{propo_R0_describe}
Let $n\geq 4$ and $C_0=\left[c_{i,j}\right]_{1\leq i,j\leq n}\in R_0^{(n}(p)$. On the one hand, $C_0\neq pI_n$. On the other hand, for all positive integers $i,j,k,\ell$, one has 
\begin{eqnarray*}
1\leq i<k<j<\ell\leq n & \Rightarrow & c_{i,j}c_{k,\ell}\equiv c_{i,\ell}c_{k,j}\pmod{p} \\
1\leq i<j\leq k<\ell\leq n & \Rightarrow & c_{i,j}c_{k,\ell}=0.
\end{eqnarray*}
\end{proposition}
\begin{remark}
One can check that
\begin{eqnarray*}
R_0^{(2)}(p) & = & \bigsqcup_{0<c_{1,2}<p}\left\{\begin{pmatrix}
p & c_{1,2} \\
& p
\end{pmatrix}\right\}, \\
R_0^{(3)}(p) & = & \bigsqcup_{\substack{0\leq c_{1,2},c_{1,3},c_{2,3}<p \\
c_{1,2}c_{2,3}=0 \\
\left(c_{1,2},c_{1,3},c_{2,3}\right)\neq(0,0,0)}}\left\{\begin{pmatrix}
p & c_{1,2} & c_{1,3} \\
& p & c_{2,3} \\
& & p
\end{pmatrix}\right\}.
\end{eqnarray*}
\end{remark}
\begin{proof}[\proofname{} of Proposition \ref{propo_R0_describe}]%
The fact that $C_0\neq pI_n$ is obvious since the first determinantal divisor of $C_0$, whose value is $1$, is nothing else than the greatest common divisor of the coefficients of $C_0$, which are non-negative integers strictly less than $p$.
\par
Recall that $d_2(C_0)=p$. As a consequence, $p$ divides the determinantal minors of $C_0$ of size $2$ given by
\begin{equation}\label{eq_det2_1}
c_{i,j}c_{k,\ell}-c_{i,\ell}c_{k,j}
\end{equation}
for all $1\leq i<k<j<\ell\leq n$. It also divides the determinantal divisors of $C_0$ of size $2$ given by
\begin{equation}
c_{i,j}c_{j,\ell}-c_{i,\ell}c_{j,j}=c_{i,j}c_{j,\ell}-pc_{i,\ell}
\end{equation}
for $1\leq i<j<\ell\leq n$. The fact that the prime number $p$ divides $c_{i,j}c_{j,\ell}$implies that $c_{i,j}c_{j,\ell}=0$ because the non-diagonal coefficients of $C_0$ are non-negative and strictly less than $p$. Similarly, $p$ divides the determinantal divisors of $C_0$ of size $2$ given by
\begin{equation}
c_{i,j}c_{k,\ell}-c_{i,\ell}c_{k,j}=c_{i,j}c_{k,\ell}
\end{equation}
for $1\leq i<j<k<\ell\leq n$, such that $c_{i,j}c_{k,\ell}=0$ too.
\end{proof}
%................................................................................................
\section{End of the proof of Theorem \ref{coro_A}}\label{sec_proof_eq}%
In this section, $n\geq 4$. First of all, we need the following intermediate result.
\begin{proposition}\label{propo_R0}
Let $n\geq 4$. Let $C_0=\left[c_{i,j}\right]_{1\leq i,j\leq n}$ in $R_0^{(n)}(p)$. If
\begin{equation}\label{eq_cond}
\forall (i,j)\in\{1,\dots,n\}^2,\quad 2\leq i<j\leq n-1\Rightarrow c_{i,j}=0
\end{equation}
then
\begin{equation*}
C_0D^{(n)}(p)\in\Lambda_n\text{diag}\left(p,\underbrace{p^2,\dots,p^2}_{n-2 \text{ terms}},p^3\right)\Lambda_n.
\end{equation*}
Otherwise,
\begin{equation*}
C_0D^{(n)}(p)\in\Lambda_n\text{diag}\left(p,p,\underbrace{p^2,\dots,p^2}_{n-4 \text{ terms}},p^3,p^3\right)\Lambda_n.
\end{equation*}
In addition,
\begin{equation*}
\text{card}\left(\left\{C_0\in R_0^{(n)}(p), C_0D^{(n)}(p)\in\Lambda_n\text{diag}\left(p,\underbrace{p^2,\dots,p^2}_{n-2 \text{ terms}},p^3\right)\Lambda_n\right\}\right)=2p^{n-1}-p-1
\end{equation*}
and
\begin{multline*}
\text{card}\left(\left\{C_0\in R_0^{(n)}(p), C_0D^{(n)}(p)\in\Lambda_n\text{diag}\left(p,p,\underbrace{p^2,\dots,p^2}_{n-4 \text{ terms}},p^3,p^3\right)\Lambda_n\right\}\right) \\
=\frac{p^2\left((n-3)p^{n-2}-(n-2)p^{n-3}+1\right)}{p-1}.
\end{multline*}
\end{proposition}
\begin{remark}
One can check that the previous proposition remains valid when $n=3$, in which case
\begin{equation*}
C_0D^{(3)}(p)\in\Lambda_3\text{diag}\left(p,p^2,p^3\right)\Lambda_3
\end{equation*}
for all matrix $C_0\in R_0^{(3)}(p)$. When $n=2$,
\begin{equation*}
C_0D^{(2)}(p)\in\Lambda_2\text{diag}\left(p,p^3\right)\Lambda_2
\end{equation*}
for all matrix $C_0\in R_0^{(2)}(p)$. 
\end{remark}
\begin{proof}[\proofname{} of Proposition \ref{propo_R0}]%
Recall that
\begin{eqnarray*}
d_n\left(\text{diag}\left(p,\underbrace{p^2,\dots,p^2}_{n-2 \text{ terms}},p^3\right)\right) & = & \left(p,p^3,\dots,\underbrace{p^{2k-1}}_{\text{$k$'th term}},\dots,p^{2n-5},p^{2n-3},p^{2n}\right), \\
d_n\left(\text{diag}\left(p,p,\underbrace{p^2,\dots,p^2}_{n-4 \text{ terms}},p^3,p^3\right)\right) & = & \left(p,p^2,\dots,\underbrace{p^{2k-2}}_{\text{$k$'th term}},\dots,p^{2n-6},p^{2n-3},p^{2n}\right), \\
d_n(C_0) & = & d_n(p)=\left(1,p,\dots,\underbrace{p^{\ell-1}}_{\text{$\ell$'th term}},\dots,p^{n-2},p^n\right) 
\end{eqnarray*}
for $2\leq k\leq n-2$ and $2\leq\ell\leq n-1$.
\par
Obviously, $d_{1}(C_0D^{(n)}(p))=p$ and $d_n(C_0D^{(n)}(p))=p^{2n}$.
\par
Let us show that $d_{n-1}(C_0D^{(n)}(p))=p^{2n-3}$. Of course, $p^{2n-3}$ is a determinantal minor of $C_0D^{(n)}(p)$ of size $n-1$ such that it remains to show that the other determinantal minors of $C_0D^{(n)}(p)$ of size $n-1$ are all divisible by $p^{2n-3}$. Let $\omega=\{1,\dots,n\}\setminus\{i_0\}$ and $\tau=\{1,\dots,n\}\setminus\{j_0\}$ two elements in $I_{n-1,n}$ (see \eqref{eq_def_dk} for the notations used). By the Cauchy-Binet formula,
\begin{align*}
\left(C_0D^{(n)}(p)\right)\left(\omega,\tau\right) & =\sum_{\alpha\in I_{n-1,n}}C_0\left(\omega,\alpha\right)D^{(n)}(p)\left(\alpha,\tau\right) \\
& =C_0\left(\omega,\tau\right)D^{(n)}(p)\left(\tau,\tau\right)
\end{align*}
since $D^{(n)}(p)$ is a diagonal matrix. If $j_0=1$ then $C_0\left(\omega,\tau\right)$ is divisible by $p^{n-2}$, since $d_{n-1}(C_0)=p^{n-2}$, and $D^{(n)}(p)\left(\tau,\tau\right)=p^n$. If $2\leq j_0\leq n-1$ then $C_0\left(\omega,\tau\right)$ is divisible by $p^{n-2}$ and $D^{(n)}(p)\left(\tau,\tau\right)=p^{n-1}$. The only remaining case is when $j_0=n$. The minor obtained when erasing the $i_0$'th row and the $n$'th column of $C_0D^{(n)}(p)$ has its last row equal to $0$ but when $i_0=n$, in which case
\begin{equation*}
\left(C_0D^{(n)}(p)\right)\left(\omega,\tau\right)=p^{2n-3}.
\end{equation*} 
\par
Let $2\leq k\leq n-2$. Of course, $p^{2k-1}$  is a determinantal minor of $C_0D^{(n)}(p)$ of size $k$. Then, by Lemma \ref{lemma_submatrices_3}, all the integers
\begin{equation*}
p^{2k-2}c_{i,j}
\end{equation*}
for $2\leq i<j\leq n-1$ also belong to the list of determinantal minors of $C_0D^{(n)}(p)$ of size $k$. Let $\omega=\{i_1,\dots,i_k\}$ with $1\leq i_1<\dots<i_k\leq n$ and $\tau=\{j_1,\dots,j_k\}$ with $1\leq j_1<\dots<j_k\leq n$ two elements in $I_{k,n}$. Once again, by the Cauchy-Binet formula,
\begin{align*}
\left(C_0D^{(n)}(p)\right)\left(\omega,\tau\right) & =\sum_{\alpha\in I_{k,n}}C_0\left(\omega,\alpha\right)D^{(n)}(p)\left(\alpha,\tau\right) \\
& =C_0\left(\omega,\tau\right)D^{(n)}(p)\left(\tau,\tau\right) \\
& =C_0\left(\omega,\tau\right)\times\begin{cases}
p^{k+1} & \text{if $2\leq j_1<\dots<j_{k-1}\leq n-1<j_k=n$,} \\
p^k & \text{if $2\leq j_1<\dots<j_{k}\leq n-1$,} \\
p^k & \text{if $1=j_1<j_2\dots<j_{k-1}<j_k=n$,} \\
p^{k-1} & \text{if $1=j_1<j_2\dots<j_{k}\leq n-1$.}
\end{cases}
\end{align*}
$C_0\left(\omega,\tau\right)$ being divisible by $p^{k-1}$, since $d_{k}(C_0)=p^{k-1}$, all these determinantal minors are divisible by $p^{2k-1}$ except a priori when $1=j_1<j_2\dots<j_{k}\leq n-1$. Let us investigate this last case. First of all,
\begin{align*}
C_0\left(\omega,\tau\right) & =\sum_{\sigma_\in\sigma_k}\epsilon(\sigma)c_{i_{\sigma(1)},1}c_{i_{\sigma(2)},j_2}\dots c_{i_{\sigma(k)},j_k} \\
& =\sum_{\substack{\sigma_\in\sigma_k \\
i_{\sigma(1)}=1}}\epsilon(\sigma)c_{i_{\sigma(1)},1}c_{i_{\sigma(2)},j_2}\dots c_{i_{\sigma(k)},j_k} \\
& =\begin{cases}
p\sum_{\substack{\sigma_\in\sigma_k \\
\sigma(1)=1}}\epsilon(\sigma)c_{i_{\sigma(2)},j_2}\dots c_{i_{\sigma(k)},j_k} & \text{if $i_1=1$,} \\
0 & \text{otherwise}
\end{cases}
\end{align*}
where $\sigma_k$ stands for the permutation group on $k$ letters and since the condition $i_{\sigma(1)}=1$ is equivalent to $i_1=\sigma(1)=1$. We can focus on the case $i_1=1$, in which case
\begin{equation*}
C_0(\omega,\tau)=\sum_{L=0}^{k-1}p^{1+L}\sum_{\substack{\sigma_\in\sigma_k \\
\sigma(1)=1 \\
\forall\ell\in\{2,\dots,k\}, i_{\sigma(\ell)}\leq j_\ell \\
\text{card}\left(\left\{\ell\in\{2,\dots,k\}, i_{\sigma(l)}=j_\ell\right\}\right)=L}}\epsilon(\sigma)\prod_{\substack{2\leq\ell\leq k \\
i_{\sigma(\ell)}\neq j_\ell}}c_{i_{\sigma(\ell)},j_\ell}
\end{equation*}
is a polynomial in a subset of 
\begin{equation*}
c_{i,j},\quad 2\leq i<j\leq n-1
\end{equation*}
divisible by $p^{k-1}$, since $d_{k}(C_0)=p^{k-1}$, whose constant term is divisible by $p^k$. One can now conclude as follows. If \eqref{eq_cond} holds then $d_k\left(C_0D^{(n)}(p)\right)$ is the greatest common divisor of $0$, $p^{2k-1}$ and of a finite list of integers divisible by $p^{2k-1}$ such that
\begin{equation*}
d_k\left(C_0D^{(n)}(p)\right)=p^{2k-1}.
\end{equation*}
If \eqref{eq_cond} does not hold then $d_k\left(C_0D^{(n)}(p)\right)$ is the greatest common divisors of $p^{2k-1}$, of the integers $p^{2k-2}c_{i,j}$, $2\leq i<j\leq n-1$, and of a finite list of integers divisible by $p^{2k-2}$ such that
\begin{equation*}
d_k\left(C_0D^{(n)}(p)\right)=p^{2k-2}.
\end{equation*}
\par
Let us compute the first cardinality, say $c_0^{(n)}(p)$ , given in the previous proposition. The set
\begin{equation*}
\left\{C_0\in R_0^{(n)}(p), \forall (i,j)\in\{1,\dots,n\}^2,\quad 2\leq i<j\leq n-1\Rightarrow c_{i,j}=0\right\}
\end{equation*}
can be decomposed into the disjoint union of the three following sets.
\begin{itemize}
\item
The set of matrices $C_0$ in $R_0^{(n)}(p)$ satisfying \eqref{eq_cond} and $c_{1,2}\neq 0$, $c_{n-1,n}=0$, which implies that
\begin{equation*}
c_{2,n}=\dots=c_{n-2,n}=0.
\end{equation*}
There are $(p-1)p^{n-2}$ such matrices.
\item
The set of matrices $C_0$ in $R_0^{(n)}(p)$ satisfying \eqref{eq_cond} and $c_{1,2}=0$, $c_{n-1,n}\neq 0$, which implies that
\begin{equation*}
c_{1,3}=\dots=c_{1,n-1}=0.
\end{equation*}
There are $(p-1)p^{n-2}$ such matrices.
\item
The set of matrices $C_0$ in $R_0^{(n)}(p)$ satisfying \eqref{eq_cond} and $c_{1,2}=c_{n-1,n}=0$, which can be identified to the set of matrices $C_0$ in $R_0^{(n-1)}(p)$ satisfying \eqref{eq_cond}, by erasing the diagonal of zeros above the main diagonal. There are $c_0^{(n-1)}(p)$ such matrices.
\end{itemize}
In total,
\begin{equation*}
c_0^{(n)}(p)=2(p-1)p^{n-2}+c_0^{(n-1)}(p).
\end{equation*}
One can conclude by induction on $n\geq 4$. If the formula holds for $n\geq 4$ then
\begin{equation*}
c_0^{(n+1)}(p)=2(p-1)p^{n-1}+2p^{n-1}-p-1=2p^n-p-1.
\end{equation*}
Let us briefly check that $c_0^{(4)}(p)=2p^3-p-1$. If $C_0$ in $R_0^{(4)}(p)$ satisfies \eqref{eq_cond} then five cases can occur.
\begin{itemize}
\item
$c_{1,2}=c_{1,3}=c_{1,4}=c_{2,4}=0$ and $c_{3,4}\neq 0$. There are $p-1$ such matrices.
\item
$c_{1,2}=c_{1,3}=c_{1,4}=0$ and $c_{2,4}\neq 0$. There are $p(p-1)$ such matrices.
\item
$c_{1,2}=c_{1,3}=0$ and $c_{1,4}\neq 0$. There are $p^2(p-1)$ such matrices.
\item
$c_{1,2}=c_{2,4}=c_{3,4}=0$ and $c_{1,3}\neq 0$. There are $p(p-1)$ such matrices.
\item
$c_{2,4}=c_{3,4}=0$ and $c_{1,2}\neq 0$. There are $p^2(p-1)$ such matrices.
\end{itemize}
\par
The computation of the second cardinality is a consequence of Proposition \ref{propo_representatives}, which gives the cardinal of $R_0^{(n)}(p)$.
\end{proof}
The following lemma has been used in the previous proof.
\begin{lemma}\label{lemma_submatrices_3}
Let $n\geq 4$ and $2\leq k\leq n-2$. Let $C=\left[c_{i,j}\right]_{1\leq i,j\leq n}$ be an upper-triangular matrix with integer coefficients satisfying
\begin{equation}\label{eq_rec_1}
\forall i\in\{1,\dots,n\},\quad c_{i,i}=p
\end{equation}
and
\begin{equation}\label{eq_rec_2}
1\leq i<j\leq k<\ell\leq n\Rightarrow c_{i,j}c_{k,\ell}=0
\end{equation}
for all positive integer $i,j,k,\ell$. Let
\begin{equation}\label{eq_index}
2\leq i_0<j_0\leq n-1.
\end{equation}
Then, there exists $\omega_{i_0,j_0}$, $\tau_{i_0,j_0}$ in $I_{k,n}$ and $\epsilon_{i_0,j_0}=\pm 1$ such that
\begin{equation*}
\left(CD^{(n)}(p)\right)\left(\omega_{i_0,j_0},\tau_{i_0,j_0}\right)=\epsilon_{i_0,j_0}p^{2k-2}c_{i_0,j_0}.
\end{equation*}
\end{lemma}
\begin{proof}[\proofname{} of Lemma \ref{lemma_submatrices_3}]%
Let $\omega_0=\{1,i_2,\dots,i_k\}$ with $2\leq i_2$ and $\tau_0=\{1,j_2,\dots,j_k\}$ with $2\leq j_2$ and $j_k\leq n-1$ in $I_{k,n}$ satisfying
\begin{equation*}
\exists(u_0,v_0)\in\{2,\dots,k\}^2,\quad (i_{u_0},j_{v_0})=(i_0,j_0)
\end{equation*}
and
\begin{equation*}
\forall\ell\in\{2,\dots,k\}\setminus\{v_0\},\quad i_\ell=j_{\varphi_0(\ell)}
\end{equation*}
for some permutation $\varphi_0$ in $\sigma_{k-2}$. Such a choice is possible by \eqref{eq_index}. Recall the Cauchy-Binet formula stated in the proof of Proposition \ref{propo_R0}, namely
\begin{align}\label{eq_Cauchy_Binet}
\left(CD^{(n)}(p)\right)\left(\omega_0,\tau_0\right) & =p^{k}\sum_{\sigma\in\sigma_{k-1}}\epsilon(\sigma)c_{i_{\sigma(2)},j_2}\dots c_{i_{\sigma(v_0)},j_{v_0}}\dots c_{i_{\sigma(k)},j_k} \\
 & =p^{k}\sum_{\sigma\in\sigma_{k-1}}\epsilon(\sigma)c_{j_{\varphi_0(\sigma(2))},j_2}\dots c_{i_{\sigma(v_0)},j_{v_0}}\dots c_{j_{\varphi_0(\sigma(k))},j_k}
\end{align}
\par
Obviously, the contribution to the previous sum of the permutation $\sigma_0$ in $\sigma_{k-1}$ given by $\sigma_0(v_0)=u_0$ and
\begin{equation*}
\sigma_0(\ell)=\varphi_0^{-1}(\ell)
\end{equation*}
for $\ell$ in $\{2,\dots,k\}\setminus\{v_0\}$ equals
\begin{equation*}
\epsilon(\sigma_0)p^{2k-2}c_{i_0,j_0}
\end{equation*}
by \eqref{eq_rec_1}.
\par
Let us show that all the other terms vanish. Let $\sigma\neq\sigma_0$ in $\sigma_{k-1}$. One can assume that $i_{\sigma(v_0)}\leq j_{v_0}=j_0$ and
\begin{equation*}
j_{\varphi_0(\sigma(\ell))}\leq j_\ell
\end{equation*}
for $\ell\in\{2,\dots,k\}\setminus\{v_0\}$ since otherwise, the contribution trivially vanishes, $C$ being upper-triangular. This immediately implies that
\begin{equation*}
j_{\varphi_0(\sigma(\ell))}=j_\ell \text{ and } \sigma(\ell)=\varphi_0^{-1}(\ell)
\end{equation*}
for $2\leq\ell\leq v_0-1$ since $2\leq j_2<\dots<j_{v_0}=j_0<\dots<j_k$. Thus, by \eqref{eq_rec_1}, the contribution of $\sigma$ equals
\begin{equation*}
p^{k+v_0-2}\epsilon(\sigma)c_{i_{\sigma(v_0)},j_{v_0}}c_{j_{\varphi_0(\sigma(v_0+1))},j_{v_0+1}}\dots c_{j_{\varphi_0(\sigma(k))},j_k}.
\end{equation*}
Now, $\sigma$ being different from $\sigma_0$, there exists $\ell\in\{v_0+1,\dots,k\}$ satisfying $j_{\varphi_0(\sigma(\ell))}<j_\ell$. Let $\ell_0$ be the minimum of these integers $\ell$. One has
\begin{equation*}
j_{\varphi_0(\sigma(\ell))}=j_\ell \text{ and } \sigma(\ell)=\varphi_0^{-1}(\ell)
\end{equation*}
for $2\leq\ell\in\{2,\dots,\ell_0-1\}\setminus\{v_0\}$ since $2\leq j_2<\dots<j_{v_0}=j_0<\dots<j_k$ such that
\begin{equation*}
j_{\ell_0}>j_{\varphi_0(\sigma(\ell_0))}\Rightarrow j_{\varphi_0(\sigma(\ell_0))}=j_{v_0}. 
\end{equation*}
Consequently, the contribution of $\sigma$ equals
\begin{equation*}
p^{k+\ell_0-3}\epsilon(\sigma)c_{i_{\sigma(v_0)},j_{v_0}}c_{j_{v_0},j_{\ell_0}}\dots c_{j_{\varphi_0(\sigma(k))},j_k}=0
\end{equation*}
by \eqref{eq_rec_2}.
\end{proof}
Let us now complete the proof of Theorem \ref{coro_A}.
\begin{proof}[\proofname{} of Theorem \ref{coro_A}]%
By \eqref{eq_product},
\begin{equation*}
\pi^{(n)}(p)\ast\pi^{(n)}(p)=\sum_{\Lambda_n h\Lambda_n\subset\pi^{(n)}(p)\pi^{(n)}(p)}m_n(h;p)\Lambda_n h\Lambda_n
\end{equation*}
where $h\in GL_n(\Q)$ ranges over a system of representatives of the $\Lambda_n$ right cosets contained in the set
\begin{equation*}
\pi^{(n)}(p)\pi^{(n)}(p)
\end{equation*}
and
\begin{eqnarray*}
m_n(h;p) & \coloneqq & \frac{\text{deg}\left(D^{(n)}(p)\right)}{\text{deg}(h)}c_n(h;p), \\
c_n(h;p) & \coloneqq & \text{card}\left(\left\{C\in R^{(n)}(p),CD^{(n)}(p)\in\pi^{(n)}(p)\right\}\right).
\end{eqnarray*}
Recall that
\begin{equation}\label{eq_deg_fund}
\text{deg}\left(D^{(n)}(p)\right)=p\frac{\varphi_n(p)}{\varphi_1(p)^2\varphi_{n-2}(p)}=p\frac{\left(p^{n-1}-1\right)\left(p^n-1\right)}{(p-1)^2}
\end{equation}
by \eqref{eq_main_degree}.
\par
Let us determine the different matrices $h$ occuring in this decomposition.
\par
If $C_1$ in $R_1^{(n)}(p)$ then we have already seen that
\begin{equation*}
C_1=\text{diag}\left(p^{\delta_1},\dots,p^{\delta_n}\right)C_1^\prime
\end{equation*}
with $C_1^\prime$ an upper-triangular matrix in $\Lambda_n$ and $0\leq\delta_1,\dots,\delta_n\leq 2$ with
\begin{equation*}
\text{card}\left(\left\{i\in\{1,\dots,n\}, \delta_i=1\right\}\right)=\text{card}\left(\left\{i\in\{1,\dots,n\}, \delta_i=2\right\}\right)=1.
\end{equation*}
As a consequence,
\begin{align*}
C_1D^{(n)}(p) & =\text{diag}\left(p^{\delta_1},\underbrace{p^{1+\delta_2},\dots,p^{1+\delta_{n-1}}}_{n-2 \text{ terms}},p^{2+\delta_n}\right)D^{(n)}(p)^{-1}C_1^\prime D^{(n)}(p) \\
& \in \Lambda_n\text{diag}\left(p^{\delta_1},\underbrace{p^{1+\delta_2},\dots,p^{1+\delta_{n-1}}}_{n-2 \text{ terms}},p^{2+\delta_n}\right)\Lambda_n
\end{align*}
since $D^{(n)}(p)^{-1}C_1^\prime D^{(n)}(p)$ belongs to $\Lambda_n$. Let $1\leq\alpha_1\neq\alpha_2\leq n$ the integers satisfying
\begin{equation*}
\delta_{\alpha_1}=1\quad\text{ and }\quad\delta_{\alpha_2}=2.
\end{equation*}
Let us list the different cases that can occur. \\
\noindent{\underline{First case}}: $\alpha_1=1$ and $2\leq\alpha_2\leq n-1$. In this case, one has
\begin{equation*}
C_1D^{(n)}(p)\in\Lambda_n\text{diag}\left(1,\underbrace{p^{2},\dots,p^{2}}_{n-3 \text{ terms}},p^{3},p^3\right)\Lambda_n.
\end{equation*}
The number of such matrices $C_1$ is
\begin{equation}\label{eq_c_1}
\sum_{2\leq\alpha_2\leq n-1}p^{n+\alpha_2-2}=p^n\frac{p^{n-2}-1}{p-1}.
\end{equation}
\noindent{\underline{Second case}}: $\alpha_1=1$ and $\alpha_2=n$. In this case, one has
\begin{equation*}
C_1D^{(n)}(p)\in\Lambda_n\text{diag}\left(1,\underbrace{p^{2},\dots,p^{2}}_{n-2 \text{ terms}},p^4\right)\Lambda_n.
\end{equation*}
The number of such matrices $C_1$ is
\begin{equation}\label{eq_c_2}
p^{2n-2}.
\end{equation}
\noindent{\underline{Third case}}: $2\leq\alpha_1\leq n-1$ and $\alpha_2=1$. In this case, one has
\begin{equation*}
C_1D^{(n)}(p)\in\Lambda_n\text{diag}\left(p,\underbrace{p^{2},\dots,p^{2}}_{n-2 \text{ terms}},p^{3}\right)\Lambda_n.
\end{equation*}
The number of such matrices $C_1$ is
\begin{equation}\label{eq_c_3}
\sum_{2\leq\alpha_1\leq n-1}p^{n-\alpha_1}=p\frac{p^{n-2}-1}{p-1}.
\end{equation}
\noindent{\underline{Fourth case}\footnote{Note that this case does not occur if $n<4$.}}: $2\leq\alpha_1\neq\alpha_2\leq n-1$. In this case, one has
\begin{equation*}
C_1D^{(n)}(p)\in\Lambda_n\text{diag}\left(p,p,\underbrace{p^{2},\dots,p^{2}}_{n-4 \text{ terms}},p^{3},p^3\right)\Lambda_n.
\end{equation*}
The number of such matrices $C_1$ is
\begin{multline}\label{eq_c_4}
\sum_{2\leq\alpha_1<\alpha_2\leq n-1}p^{n-1+\alpha_2-\alpha_1}+\sum_{2\leq\alpha_2<\alpha_1\leq n-1}p^{n-1+\alpha_2-\alpha_1} \\
=\frac{p^2\left(p^{2(n-2)}-(n-2)p^{n-1}+2(n-3)p^{n-2}-(n-2)p^{n-3}+1\right)}{(p-1)^2}
\end{multline}
since it is straightforward to check that
\begin{equation*}
\sum_{2\leq\alpha_1<\alpha_2\leq n-1}x^{\alpha_2-\alpha_1}=\frac{x\left(x^{n-2}-(n-2)x+n-3\right)}{(x-1)^2}
\end{equation*}
for all real number $x\neq 1$. \\
\noindent{\underline{Fifth case}}: $2\leq\alpha_1\leq n-1$ and $\alpha_2=n$. In this case, one has
\begin{equation*}
C_1D^{(n)}(p)\in\Lambda_n\text{diag}\left(p,p,\underbrace{p^{2},\dots,p^{2}}_{n-3 \text{ terms}},p^4\right)\Lambda_n.
\end{equation*}
The number of such matrices $C_1$ is
\begin{equation}\label{eq_c_5}
\sum_{2\leq\alpha_1\leq n-1}p^{2n-1-\alpha_1}=p^n\frac{p^{n-2}-1}{p-1}.
\end{equation}
\noindent{\underline{Sixth case}}: $\alpha_1=n$ and $\alpha_2=1$. In this case, one has
\begin{equation*}
C_1D^{(n)}(p)\in\Lambda_n\text{diag}\left(\underbrace{p^{2},\dots,p^{2}}_{n \text{ terms}}\right)\Lambda_n=\Lambda_np^2I_n\Lambda_n.
\end{equation*}
The number of such matrices $C_1$ is
\begin{equation}\label{eq_c_6}
1.
\end{equation}
\noindent{\underline{Seventh case}}: $\alpha_1=n$ and $2\leq\alpha_2\leq n-1$. In this case, one has
\begin{equation*}
C_1D^{(n)}(p)\in\Lambda_n\text{diag}\left(p,\underbrace{p^{2},\dots,p^{2}}_{n-2 \text{ terms}},p^3\right)\Lambda_n.
\end{equation*}
The number of such matrices $C_1$ is
\begin{equation}\label{eq_c_7}
\sum_{2\leq\alpha_2\leq n-1}p^{\alpha_2-1}=p\frac{p^{n-2}-1}{p-1}.
\end{equation}
\par
If $C_0$ in $R_0^{(n)}(p)$ then two cases can occur by Proposition \ref{propo_R0}. \\
\noindent{\underline{Eighth case}}: $\forall (i,j)\in\{1,\dots,n\}^2, 2\leq i<j\leq n\Rightarrow c_{i,j}=0$.
In this case,
\begin{equation*}
C_0D^{(n)}(p)\in\Lambda_n\text{diag}\left(p,\underbrace{p^2,\dots,p^2}_{n-2 \text{ terms}},p^3\right)\Lambda_n
\end{equation*}
and the number of such matrices is
\begin{equation}\label{eq_c_8}
2p^{n-1}-p-1.
\end{equation}
\noindent{\underline{Nineth case}}: $\exists (i,j)\in\{1,\dots,n\}^2, 2\leq i<j\leq n \text{ and } c_{i,j}\neq 0$. In this case,
\begin{equation*}
C_0D^{(n)}(p)\in\Lambda_n\text{diag}\left(p,p,\underbrace{p^2,\dots,p^2}_{n-4 \text{ terms}},p^3,p^3\right)\Lambda_n
\end{equation*}
and the number of such matrices is
\begin{equation}\label{eq_c_9}
\frac{p^2\left((n-3)p^{n-2}-(n-2)p^{n-3}+1\right)}{p-1}.
\end{equation}
\par
In particular, we have just proved that
\begin{multline}\label{eq_inter}
\pi^{(n)}(p)\ast\pi^{(n)}(p)=m_n(1;p)\Lambda_np^2I_n\Lambda_n+m_n(2;p)\Lambda_n\text{diag}\left(p,\underbrace{p^2,\dots,p^2}_{n-2 \text{ terms}},p^3\right)\Lambda_n \\
+m_n(3;p)\Lambda_n\text{diag}\left(1,\underbrace{p^2,\dots,p^2}_{n-3 \text{ terms}},p^3,p^3\right)\Lambda_n+m_n(4;p)\Lambda_n\text{diag}\left(1,\underbrace{p^2,\dots,p^2}_{n-2 \text{ terms}},p^4\right)\Lambda_n \\
+m_n(5;p)\Lambda_n\text{diag}\left(p,p,\underbrace{p^2,\dots,p^2}_{n-3 \text{ terms}},p^4\right)\Lambda_n+m_n(6;p)\Lambda_n\text{diag}\left(p,p,\underbrace{p^2,\dots,p^2}_{n-4 \text{ terms}},p^3,p^3\right)\Lambda_n.
\end{multline}
where
\begin{eqnarray*}
m_n(1;p) & \coloneqq & m_n\left(p^2I_n;p\right), \\
m_n(2;p) & \coloneqq & m_n\left(\text{diag}\left(p,\underbrace{p^2,\dots,p^2}_{n-2 \text{ terms}},p^3\right);p\right), \\
m_n(3;p) & \coloneqq & m_n\left(\text{diag}\left(1,\underbrace{p^2,\dots,p^2}_{n-3 \text{ terms}},p^3,p^3\right);p\right)
\end{eqnarray*}
and
\begin{eqnarray*}
m_n(4;p) & \coloneqq & m_n\left(\text{diag}\left(1,\underbrace{p^2,\dots,p^2}_{n-2 \text{ terms}},p^4\right);p\right), \\
m_n(5;p) & \coloneqq & m_n\left(\text{diag}\left(p,p,\underbrace{p^2,\dots,p^2}_{n-3 \text{ terms}},p^4\right);p\right), \\
m_n(6;p) & \coloneqq & m_n\left(\text{diag}\left(p,p,\underbrace{p^2,\dots,p^2}_{n-4 \text{ terms}},p^3,p^3\right);p\right).
\end{eqnarray*}
\par
One has,
\begin{align*}
m_n(1;p) & =\frac{\text{deg}\left(D^{(n)}(p)\right)}{\text{deg}(p^2I_n)}c_n\left(p^2I_n;p\right) \\
& =p\frac{\left(p^{n-1}-1\right)\left(p^n-1\right)}{(p-1)^2}
\end{align*}
by \eqref{eq_deg_fund} and \eqref{eq_c_6} since $\text{deg}(p^2I_n)=1$.
\par
Then,
\begin{align*}
m_n(2;p) & =\frac{\text{deg}\left(D^{(n)}(p)\right)}{\text{diag}\left(p,\underbrace{p^2,\dots,p^2}_{n-2 \text{ terms}},p^3\right)}c_n\left(\text{diag}\left(p,\underbrace{p^2,\dots,p^2}_{n-2 \text{ terms}},p^3\right);p\right) \\
& =c_n\left(\text{diag}\left(p,\underbrace{p^2,\dots,p^2}_{n-2 \text{ terms}},p^3\right);p\right) \\
& =2p\frac{p^{n-2}-1}{p-1}+2p^{n-1}-p-1 \\
& =\frac{2p^n-p^2-2p+1}{p-1}
\end{align*}
by \eqref{eq_dilate}, \eqref{eq_c_3}, \eqref{eq_c_7}, \eqref{eq_c_8}.
\par
Let us compute simultaneously the values of $m_n(3;p)$ and $m_n(4;p)$. On the one hand, applying the map $\Psi$ (see \eqref{eq_psi}) to \eqref{eq_inter}, one gets
\begin{multline*}
\pi_{n-2,1}^{(n-1)}(p)\ast\pi_{n-2,1}^{(n-1)}(p)=m_n(3;p)\Lambda_n\text{diag}\left(\underbrace{p^2,\dots,p^2}_{n-3 \text{ terms}},p^3,p^3\right)\Lambda_n \\
+m_n(4;p)\Lambda_n\text{diag}\left(\underbrace{p^2,\dots,p^2}_{n-2 \text{ terms}},p^4\right)\Lambda_n.
\end{multline*}
On the other hand, by \cite[Lemma 2.18 Page 115]{MR1349824}, one gets
\begin{align*}
\pi_{n-2,1}^{(n-1)}(p)\ast\pi_{n-2,1}^{(n-1)}(p) & =\Lambda_np^2I_n\Lambda_n\ast\pi_{1}^{(n-1)}(p)\ast\pi_{1}^{(n-1)}(p) \\
& =\Lambda_np^2I_n\Lambda_n\ast\left(\pi_{0,1}^{(n-1)}(p)+(p+1)\pi_{2,0}^{(n-1)}(p)\right) \\
& =\Lambda_n\text{diag}\left(\underbrace{p^2,\dots,p^2}_{n-2 \text{ terms}},p^4\right)\Lambda_n \\
& +(p+1)\Lambda_n\text{diag}\left(\underbrace{p^2,\dots,p^2}_{n-3 \text{ terms}},p^3,p^3\right)\Lambda_n
\end{align*}
by \eqref{eq_ex_multi}. Distinct $\Lambda_n$ double cosets being linearly independent by \cite[Lemma 1.5 Page 96]{MR1349824}, we get
\begin{eqnarray*}
m_n(3;p) & = & p+1, \\
m_n(4;p) & = & 1.
\end{eqnarray*}
Then,
\begin{align*}
\text{deg}\left(\text{diag}\left(1,\underbrace{p^2,\dots,p^2}_{n-3 \text{ terms}},p^3,p^3\right)\right) & =\frac{\text{deg}\left(D^{(n)}(p)\right)}{m_n(3;p)}c_n\left(\text{diag}\left(1,\underbrace{p^2,\dots,p^2}_{n-3 \text{ terms}},p^3,p^3\right);p\right) \\
& =p^{n+1}\frac{\left(p^{n-2}-1\right)\left(p^{n-1}-1\right)\left(p^{n}-1\right)}{(p-1)^2\left(p^2-1\right)}
\end{align*}
by \eqref{eq_deg_fund} and \eqref{eq_c_1}. Similarly,
\begin{align*}
\text{deg}\left(\text{diag}\left(1,\underbrace{p^2,\dots,p^2}_{n-2 \text{ terms}},p^4\right)\right) & =\frac{\text{deg}\left(D^{(n)}(p)\right)}{m_n(4;p)}c_n\left(\text{diag}\left(1,\underbrace{p^2,\dots,p^2}_{n-2 \text{ terms}},p^4\right);p\right) \\
& =p^{2n-1}\frac{\left(p^{n-1}-1\right)\left(p^{n}-1\right)}{(p-1)^2}
\end{align*}
by \eqref{eq_deg_fund} and \eqref{eq_c_2}.
\par
Let us consider $m_n(5;p)$. Firstly, let us compute the value of
\begin{equation*}
\text{deg}\left(\text{diag}\left(p,p,\underbrace{p^2,\dots,p^2}_{n-3 \text{ terms}},p^4\right)\right)=\text{deg}\left(\text{diag}\left(1,1,\underbrace{p,\dots,p}_{n-3 \text{ terms}},p^3\right)\right)
\end{equation*}
by \eqref{eq_dilate}. This is done by a semi-explicit computation of
\begin{equation*}
\pi_{n-2}^{(n)}(p)\ast\pi_{0,1}^{(n)}(p)=\sum_{\Lambda_nh\Lambda_n\subset\pi_{n-2}^{(n)}(p)\pi_{0,1}^{(n)}(p)}m\left(D_{n-2}^{(n)}(p),D_{0,1}^{(n)}(p);h\right)\Lambda_nh\Lambda_n
\end{equation*}
where $h\in GL_n(\Q)$ ranges over a system of representatives of the $\Lambda_n$ right cosets contained in the set
\begin{equation*}
\pi_{n-2}^{(n)}(p)\pi_{0,1}^{(n)}(p)
\end{equation*}
and
\begin{equation*}
m\left(D_{n-2}^{(n)}(p),D_{0,1}^{(n)}(p);h\right)=\frac{\text{deg}\left(D_{0,1}^{(n)}(p)\right)}{\text{deg}(h)}\text{card}\left(\left\{C\in R_{1,1,\underbrace{p,\dots,p}_{n-2}}, CD_{0,1}^{(n)}(p)\in\Lambda_nh\Lambda_n\right\}\right)
\end{equation*}
where $R_{1,1,\underbrace{p,\dots,p}_{n-2}}$ is the complete system of representatives for the distinct $\Lambda_n$ right cosets of $\pi_{n-2}^{(n)}(p)$ modulo $\Lambda_n$ given by the set of upper-triangular column reduced matrices $C$ satisfying
\begin{equation}\label{eq_expos}
\forall i\in\left\{1,\dots,n\right\},\quad c_{i,i}\in\{1,p\},
\end{equation}
\begin{equation}\label{eq_coef_1}
\text{card}\left(\left\{i\in\{1,\dots,n\}, c_{i,i}=1\right\}\right)=2
\end{equation}
and
\begin{equation}\label{eq_divis}
\forall i\in\left\{1,\dots,n-1\right\}, p\mid c_{i,i}\Rightarrow\forall j\in\left\{i+1,\dots,n\right\},\quad c_{i,j}=0
\end{equation}
according to \cite[Lemma 2.18 Page 115]{MR1349824}. Let $C$ be an element of $R_{1,1,\underbrace{p,\dots,p}_{n-2}}$ and let $1\leq\alpha_1<\alpha_2\leq n$ be the indices of the diagonal elements of $C$ equal to $1$ by \eqref{eq_coef_1}. By \eqref{eq_expos} and \eqref{eq_divis}, $C$ can be decomposed into
\begin{equation*}
C=\text{diag}\left(p^{\delta_1},\dots,p^{\delta_n}\right)C^\prime
\end{equation*}
for some upper-triangular matrix $C^\prime$ in $\Lambda_n$ and integers $0\leq\delta_1,\dots,\delta_n\leq 1$ such that
\begin{equation*}
CD_{0,1}^{(n)}(p)\in\begin{cases}
\Lambda_n\text{diag}\left(1,1,\underbrace{p,\dots,p}_{n-3 \text{ terms}},p^3\right)\Lambda_n & \text{if $1\leq\alpha_1<\alpha_2\leq n-1$} \\
\pi_{n-2,1}^{(n)}(p) & \text{if $1\leq\alpha_1<\alpha_2=n$.}
\end{cases}
\end{equation*}
Thus,
\begin{multline*}
\pi_{n-2}^{(n)}(p)\ast\pi_{0,1}^{(n)}(p)=m\left(D_{n-2}^{(n)}(p),D_{0,1}^{(n)}(p);D_{n-2,1}^{(n)}(p)\right)\pi_{n-2,1}^{(n)}(p) \\
+m\left(D_{n-2}^{(n)}(p),D_{0,1}^{(n)}(p);\text{diag}\left(1,1,\underbrace{p,\dots,p}_{n-3 \text{ terms}},p^3\right)\right)\Lambda_n\text{diag}\left(1,1,\underbrace{p,\dots,p}_{n-3 \text{ terms}},p^3\right)\Lambda_n.
\end{multline*}
Applying the map $\Psi^{\circ 2}$ (see \eqref{eq_psi}) to the previous equality, one gets
\begin{multline*}
\Lambda_n\text{diag}\left(\underbrace{p,\dots,p}_{n-3 \text{ terms}},p^3\right)\Lambda_n \\
=m\left(D_{n-2}^{(n)}(p),D_{0,1}^{(n)}(p);\text{diag}\left(\underbrace{p,\dots,p}_{n-3 \text{ terms}},p^3\right)\right)\Lambda_n\text{diag}\left(\underbrace{p,\dots,p}_{n-3 \text{ terms}},p^3\right)\Lambda_n
\end{multline*}
such that
\begin{equation*}
m\left(D_{n-2}^{(n)}(p),D_{0,1}^{(n)}(p);\text{diag}\left(\underbrace{p,\dots,p}_{n-3 \text{ terms}},p^3\right)\right)=1
\end{equation*}
by the linear independence of distinct $\Lambda_n$ double cosets (\cite[Lemma 1.5 Page 96]{MR1349824}). As a consequence,
\begin{align*}
\text{deg}\left(\text{diag}\left(1,1,\underbrace{p,\dots,p}_{n-3 \text{ terms}},p^3\right)\right) & =\text{deg}\left(D_{0,1}^{(n)}(p)\right)\sum_{1\leq\alpha_1<\alpha_2\leq n-1}p^{2n-1-\alpha_1-\alpha_2} \\
& =p^{n-2}\frac{\varphi_{n}(p)}{\varphi_{n-1}(p)\varphi_{1}(p)}p^{2n-1}\sum_{1\leq\alpha_1<\alpha_2\leq n-1}\left(\frac{1}{p}\right)^{\alpha_1+\alpha_2} \\
& =p^{n-2}\frac{\varphi_{n}(p)}{\varphi_{n-1}(p)\varphi_{1}(p)}p^{2n-4}\frac{\varphi_{n-1}(1/p)}{\varphi_2(1/p)\varphi_{n-3}(1/p)} \\
& =p^{n-2}\frac{\varphi_{n}(p)}{\varphi_{n-1}(p)\varphi_{1}(p)}p^2\frac{\varphi_{n-1}(p)}{\varphi_2(p)\varphi_{n-3}(p)} \\
& =p^{n+1}\frac{\varphi_{n}(p)}{\varphi_{1}(p)\varphi_{2}(p)\varphi_{n-3}(p)}
\end{align*}
by \eqref{eq_deg_fund}, \cite[Equation (2.33) Page 115]{MR1349824} and since
\begin{equation*}
\varphi_r(1/x)=(-1)^rx^{-r(r+1)/2}\varphi_r(x)
\end{equation*}
for $r\geq 1$ and $x\neq 0$ a real number. As a consequence,
\begin{align*}
m_n(5;p) & =\frac{\text{deg}\left(D^{(n)}(p)\right)}{\text{diag}\left(p,p,\underbrace{p^2,\dots,p^2}_{n-3 \text{ terms}},p^4\right)}c_n\left(\text{diag}\left(p,p,\underbrace{p^2,\dots,p^2}_{n-3 \text{ terms}},p^4\right);p\right) \\
& =\frac{\varphi_2(p)}{\varphi_1(p)^2} \\
& =p+1
\end{align*}
by \eqref{eq_c_5}.
\par
Finally, let us compute the value of $m_n(6;p)$. One has
\begin{align*}
m_n(6;p) & =\frac{\text{deg}\left(D^{(n)}(p)\right)}{\text{deg}\left(\text{diag}\left(p,p,\underbrace{p^2,\dots,p^2}_{n-4 \text{ terms}},p^3,p^3\right)\right)}c_n\left(\text{diag}\left(p,p,\underbrace{p^2,\dots,p^2}_{n-4 \text{ terms}},p^3,p^3\right);p\right) \\
& = \frac{\text{deg}\left(D^{(n)}(p)\right)}{\text{deg}\left(D_{n-4,2}^{(n)}(p)\right)}c_n\left(\text{diag}\left(p,p,\underbrace{p^2,\dots,p^2}_{n-4 \text{ terms}},p^3,p^3\right);p\right) \\
& =\frac{\left(p+1\right)^2(p-1)^2}{p^3\left(p^{n-2}-1\right)\left(p^{n-3}-1\right)}\frac{p^3\left(p^{2n-5}-p^{n-2}-p^{n-3}+1\right)}{(p-1)^2} \\
& =(p+1)^2
\end{align*}
by \eqref{eq_dilate}, \eqref{eq_main_degree}, \eqref{eq_c_4} and \eqref{eq_c_9}.
\end{proof}
%............................................................................................
\bibliographystyle{alpha}
\bibliography{biblio}
%.......................................
\end{document}